\documentclass[11pt]{article}
\PassOptionsToPackage{obeyspaces}{url}
\usepackage{amsfonts, amsmath, amssymb}

\usepackage[hidelinks]{hyperref}
\usepackage{bookmark}
\bookmarksetup{open,numbered,depth=5} 

\usepackage{enumitem} 
\setlist[enumerate]{topsep=0pt,itemsep=-1ex,partopsep=1ex,parsep=1ex}
\setlist[itemize]{topsep=0pt,itemsep=-1ex,partopsep=1ex,parsep=1ex}

\usepackage{amsthm}

\usepackage{breakurl}
\usepackage{tikz}

\usepackage{etoolbox}
\patchcmd{\thebibliography}{\leftmargin\labelwidth}{\leftmargin\labelwidth\addtolength\itemsep{-0.1\baselineskip}}{}{}

\usepackage{mathtools}

\author{Boris Bukh\thanks{Department of Mathematical Sciences, Carnegie Mellon University, Pittsburgh, PA 15213, USA\@. Supported in part by U.S.\ taxpayers through NSF grant DMS-2154063.} \and Peter van Hintum\thanks{New College, University of Oxford, UK.} \and Peter Keevash\thanks{Mathematical Institute, University of Oxford, UK. Supported by ERC Advanced Grant 883810.}}

\title{Additive Bases: Change of Domain}
\date{}

\usepackage[nameinlink]{cleveref}

\newtheorem{theorem}{Theorem}

\newtheorem{lemma}[theorem]{Lemma}

\newtheorem{conjecture}[theorem]{Conjecture}
\newtheorem{question}[theorem]{Question}

\newcommand*{\eqdef}{\stackrel{\mbox{\normalfont\tiny def}}{=}}  
\newcommand*{\veps}{\varepsilon}                                 
\DeclarePairedDelimiter\abs{\lvert}{\rvert}                      

\newcommand*{\F}{\mathbb{F}}                                     
\newcommand*{\R}{\mathbb{R}}                                     
\newcommand*{\N}{\mathbb{N}}                                     
\newcommand*{\Z}{\mathbb{Z}}                                     
\newcommand*{\Q}{\mathbb{Q}}                                     
\newcommand*{\allone}{\mathbf{1}}                                

\begin{document}
\maketitle

\begin{abstract}
We consider two questions of Ruzsa on how the minimum size 
of an additive basis $B$ of a given set $A$ depends on the domain of $B$.
To state these questions, for an abelian group $G$ and $A \subseteq D \subseteq G$ 
we write $\ell_D(A) \eqdef \min \{ |B|: B \subseteq D, \ A \subseteq B+B \}$.
Ruzsa asked how much larger can $\ell_{\Z}(A)$ be than $\ell_{\Q}(A)$ for  $A\subset\Z$, 
and how much larger can $\ell_{\N}(A)$ be than $\ell_{\Z}(A)$ for $A\subset\N$. 
For the first question we show that if $\ell_{\Q}(A) = n$ then $\ell_{\Z}(A) \le 2n$,
and that this is tight up to an additive error of at most $O(\sqrt{n})$.
For the second question, we show that if $\ell_{\Z}(A) = n$ then $\ell_{\N}(A) \le O(n\log n)$,
and this is tight up to the constant factor. We also consider these questions for higher order bases.
Our proofs use some ideas that are unexpected 
in this context, including linear algebra and Diophantine approximation.
\end{abstract}

\section{Introduction}

Many classical questions of Combinatorial Number Theory concern additive bases,
which are sets $B$ such that any natural number $n$ can be expressed as a sum
of $k$ elements of $B$, for some fixed $k$ called the order of the basis.
There is a highly developed theory for the classical bases (see Nathanson \cite{nathanson2013additive}),
where $B$ is the set of $r$th powers for some fixed $r$ (Waring's problem, resolved by Hilbert)
or the set of primes (the Schnirel'man-Goldbach Theorem). More generally, a set $B$ 
is a \emph{$k$-basis} for $A$ if every element of $A$
can be written as a sum of $k$ elements of~$B$.
For example, the Goldbach Conjecture states that the set of primes
is a $2$-basis for the set of even  numbers $n>2$,
and the Ternary Goldbach Conjecture, proved by Helfgott \cite{helfgott2013ternary},
states that the set of primes is a $3$-basis for the set of odd numbers $n>5$.
Besides having clear intrinsic interest, such questions have also spurred
many developments in Analytic Number Theory and Harmonic Analysis.

\paragraph{Additive bases of finite sets.}

A related natural research direction, initiated by Erd\H{o}s and Newman \cite{MR0453681},
considers some fixed finite set $A$ and asks for the smallest set $B$ such that $A \subseteq B+B$.
They considered a random set $A$ of size $n$ inside $\{1,\dots,N\}$,
where for $N>n^{2+c}$ they showed that typically one needs $|B| > nc/(1+c)$.
For $N=n^2$ they posed the question of whether for all $A$ of size $n$ 
one can take $|B|=o(n)$; this was proved by Alon, Bukh and Sudakov \cite{alon2009discrete},
who showed that one can take $|B| = O(n\log\log n/\log n)$.
On the other hand, it remains open (see \cite[Problem 2.8]{aimproblems} posed by Wooley)
to give good lower bounds on $|B|$ for specific natural sets $A$, such as $A = \{ i^d: 1 \le i \le n\}$. Furthermore, even for $A$ with strong multiplicative structure, 
for which it should be particularly hard to find an additive basis, there is 
a massive literature on the `sum-product phenomenon', which still falls far short
of the classical conjecture by Erd\H{o}s \cite{erdosproblems}, 
and has broad connections to a range of topics including 
growth and expansion in groups, exponential sums 
and the Kakeya Problem (see the survey by Bourgain \cite{bourgain2013around}).

In this paper, we consider two questions of Ruzsa \cite{ruzsa_question} 
on how the minimum size of an additive basis $B$ of a given set $A$ depends on the domain of $B$.
To state these questions, for an abelian group $G$ and $A \subseteq D \subseteq G$ we write
\[ \ell_D(A) \eqdef \min \{ |B|: B \subseteq D, \ A \subseteq B+B \}. \]
Ruzsa asked how much larger can $\ell_{\Z}(A)$ be than $\ell_{\Q}(A)$ for  $A\subset\Z$, 
and how much larger can $\ell_{\N}(A)$ be than $\ell_{\Z}(A)$ for $A\subset\N$. 
Here we write $\Q$ for the rationals, $\Z$ for the integers, 
and $\N\eqdef \{0,1,2,\dotsc\}$ for the natural numbers.
Our main results give fairly tight answers to both of these questions.

\paragraph{Additive bases in \texorpdfstring{$\Q$}{Q} versus \texorpdfstring{$\Z$}{Z}.}

Our results here apply more generally to additive bases of any order.
We write $kB = \{ b_1 + \dots + b_k: b_1,\dots,b_k \in B\}$ for the $k$-fold sumset of $B$.
For an abelian group $G$ and $A \subseteq D \subseteq G$ we write
\[ \ell^{(k)}_D(A) \eqdef \min \{ |B|: B \subseteq D, \ A \subseteq kB \}. \]
In the following theorems we consider sets $A$ of integers that have
$k$-bases of size $n$ over the rationals. We show that such sets $A$
have $k$-bases of size $2n$ over the integers, and that this bound
is tight up to a lower order term. Here the upper bound (\Cref{thm:QvsZupperbound})
is fairly straightforward, so the main interest is in the lower bound (\Cref{thm:QvsZlowerbound}).

\begin{theorem}\label{thm:QvsZupperbound}
Every set $A$ satisfying $\ell_{\Q}^{(k)}(A)\leq n$ also satisfies $\ell_{\Z}^{(k)}(A)\leq 2n$.
\end{theorem}

\begin{theorem}\label{thm:QvsZlowerbound}
  For all $k,n\geq 2$ there is an $A\subset \Z$ satisfying $\ell_{\Q}^{(k)}(A)\leq n$
  and $\ell_{\Z}^{(k)}(A)\geq 2n-k^4 n^{1-1/k}$.
\end{theorem}

\paragraph{Additive bases in \texorpdfstring{$\Z$}{Z} versus \texorpdfstring{$\N$}{N}.}

Next we consider sets $A$ of natural numbers 
that have $k$-bases of size $n$ over the integers;
our task here is to find small $k$-bases for $A$ over the natural numbers.
For $k=2$ we find that it is no longer possible to find such sets of linear size;
in fact, the minimum size is $\Theta(n \log n)$. The following theorems
give upper and lower bounds that match up to a constant factor, which we did not optimise.

\begin{theorem}\label{thm:NvsZupper2}
If $A \subset \N$ with $\ell_{\Z}(A)\leq n$ then $\ell_{\N}(A)\leq 3n + 2 n \log_2 n$.
\end{theorem}

\begin{theorem}\label{thm:NvsZlower2}
For all $n \ge 1$ there is $A \subset \N$ with $\ell_{\Z}(A)\leq 2n$ 
and  $\ell_{\N}(A)\geq \tfrac{1}{25} n \log_2 n - \tfrac{1}{5}n$.
\end{theorem}

For larger $k$ our results are not as precise. 
From \Cref{thm:NvsZupper2} it follows that
if $\ell^{(k)}_{\Z}(A)\leq n$ then $\ell^{(k)}_{\N}(A)\leq O( n^{\lceil \frac{k}{2}\rceil} \log n)$; indeed, if $B\subset\mathbb{Z}$ is a $k$-basis, then $\lceil \frac{k}{2}\rceil B$ is a $2$-basis.
In the following theorem we give a cubic construction,
thus showing that the power of $n$ does not need to increase with $k$.

\begin{theorem}\label{iteratedNvsZthm}
If $A \subset \N$ with $\ell_{\Z}(A)\leq n$ 
then $\ell^{(k)}_{\N}(A)\leq 16k\log(k) n^3$.
\end{theorem}

\paragraph{Ideas and techniques.}
Somewhat surprisingly, our lower bound in \Cref{thm:QvsZlowerbound} comes from linear algebra.
The key idea is to reduce the problem to a related question in a \emph{vector model}.
For example, when $k=2$ the problem in the vector model is to find $B_0,B_1 \subset \Q^n$
minimising $|B_0|+|B_1|$ subject to $B_0+B_1$ containing all vectors $e_i+e_j$ with $1 \le i \le j \le n$,
where $\{e_1,\dots,e_n\}$ is the standard basis of $\Q^n$.
If some $B_i$, say $B_0$, spans a subspace $V$ of codimension $d$
then we can find a subspace $W$ spanned by $d$ standard basis vectors $\{e_i: i \in I\}$
such that any translate of $V$ contains at most one point of $W$.
In particular, we need at least $\tbinom{d+1}{2}$ elements in $B_1$
to cover all $e_i+e_j$ with $i,j$ in $I$. Thus we can assume $d=O(\sqrt{n})$,
from which we deduce $|B_0|+|B_1| \ge 2n - O(\sqrt{n})$.

For bases in $\Z$ versus $\N$ our constructions use dyadic partitioning.
For example, when $k=2$ the problem reduces to finding a basis $B$ for sets
of the form $\{x_r-x_t: 1 \le t<r \le n\}$ for some $x_1 < \dots <x_n$.
Our construction is to take all $x_r-x_s$ where for some $j \le \log_2 n$
we have $s = 2^j \lfloor 2^{-j} r \rfloor$ or $r = 2^j \lceil 2^{-j} r \rceil$:\
it is not hard to see that any $x_r-x_t$ can be expressed 
using this set in the form $(x_r-x_s)+(x_s-x_t)$.
Thus we include in $B$ order $n$ elements at each of the $\log_2 n$ dyadic scales.
The intuition for the lower bound is that each scale requires order $n$ elements,
although this is not true literally, only in a rather subtle averaged sense
that requires some careful counting.

For higher order bases in $\Z$ versus $\N$, the key idea is to establish 
a `separation of scales' via a Diophantine approximation argument 
that finds an arithmetic progression $P$ approximating
some fixed $k$-basis $B$ over the integers.
For elements of the iterated sumset $kB$ at the scale of $P$
we can apply a variant of the dyadic partitioning construction
to find a moderately efficient $k$-basis over the natural numbers,
losing a factor of $n$ due to the exponential loss in Diophantine approximation.
We then reduce modulo the common difference of $P$,
which is chosen to reduce the number of elements from $n$ to $n-1$,
so that we can conclude by induction, losing another factor of $n$.

\paragraph{Organisation of the paper.}
We start in the next section by proving our results
comparing additive bases over the integers to those over the rationals.
In \Cref{sec:NZ2,sec:NZk} we compare additive $k$-bases
over the natural numbers to those over the integers,
firstly for $k=2$ and then for larger $k$.
We conclude with some further remarks and open problems.

\section{\texorpdfstring{$\Z$}{Z} versus \texorpdfstring{$\Q$}{Q}} \label{sec:ZQ}

Here we consider sets $A$ of integers that have $k$-bases of size $n$ over the rationals;
our task is to find small $k$-bases for $A$ over the integers. We start with the upper bound,
which is a straightforward rounding argument.

\begin{proof}[Proof of \Cref{thm:QvsZupperbound}]
Suppose $A\subset kB$ for some set $B\subset\Q$ with $|B|=n$. Let
\[
  C\eqdef\{\lfloor b\rfloor : b\in B\}\cup \{\lceil b\rceil : b\in B\}.
\]
We claim that $A\subset kC$; this will imply that $\ell_{\Z}^{(k)}(A)\leq 2n$.

Indeed, let $a\in A$ be arbitrary. Write it as $a=b_1+\dotsb+b_k$ for some $b_1,\dotsb,b_k\in B$.
Starting with $\lfloor b_1\rfloor+\dotsb+\lfloor b_k\rfloor$
change the floor signs to the ceilings, one term at a time.
Since $\lfloor b_1\rfloor+\dotsb+\lfloor b_k\rfloor \leq a$,
each change increases the sum by $0$ or $1$, and $\lceil b_1\rceil+\dotsb+\lceil b_k\rceil\geq a$,
one of the intermediate sums is actually equal to $a$. Hence $a\in kC$.
\end{proof}

\paragraph{Vector model.}

For the remainder of this section we work towards the proof of the lower bound.
As indicated in the introduction, the key idea will be to translate the problem to a vector model.
This will be achieved by \Cref{lem:vectorization} below.
We first require the following lemma on linear equations.
In the statement, we call a matrix $A$-by-$B$ if it has $A$ columns and $B$ rows.

\begin{lemma}\label{lem:matrices}
  Suppose $L_1,\dotsc,L_r$ are $A$-by-$B$ matrices, $M_1,\dotsc,M_r$ are $C$-by-$B$ matrices 
  and $t_1,\dotsc,t_r$ are vectors in $\Z^B$ such that
  for any $c\in \Z^C$ there exists an index $i\in [r]$ and vector $x\in \Z^A$ that solves the equation
  $L_ix=M_ic+t_i$. Then there exists an index $i$ such that the equation $L_{i}x=M_{i}c$ has a solution in $\Q^A$
  for every $c\in \Q^C$. 
\end{lemma}
In the conclusion of the lemma one cannot replace $\Q^A$ by $\Z^A$. For example, let $A=B=1$,  $C=2$, $t_1=t_2=t_3=0$,
$L_1x=L_2x=L_3x=2x$, $M_1(c_1,c_2)=c_1$, $M_2(c_1,c_2)=c_2$, $M_3(c_1,c_2)=c_1+c_2$.
Here for any $c=(c_1,c_2) \in \Z^2$ we can solve one of the three equations $2x=M_ic$, 
as at least one of $\{c_1,c_2,c_1+c_2\}$ is even, 
but none of these equations on its own has a solution for all $c\in\Z^2$.
\begin{proof}
For each $i\in [r]$ we consider the affine subspace
\[
  S_i\eqdef \{c\in \Q^C:\exists x\in\Q^A,\ L_ix=M_i c+t_i\}.
\]
Assume for the sake of contradiction that all $\dim_{\Q} S_i<C$. 
For any $N$ we have $\abs{S_i\cap [N]^C}\leq N^{C-1}$,
by projecting onto a co-ordinate hyperplane $e_j^\perp$ where $e_j \notin S_i-S_i$.
We deduce $\abs{\bigcup_{i\leq r} (S_i\cap [N]^C)}\leq rN^{C-1}$, which contradicts
the assumption that $\bigcup_i (S_i\cap \Z^C)=\Z^C$ if $N>r$.
\end{proof}

\begin{lemma}\label{lem:vectorization}
  Suppose $\ell_{\Z}^{(k)}(A)\leq m$ for every $A\subset \Z$ satisfying $\ell_{\Q}^{(k)}(A)\leq n$.
  Let $C=\{e_1,\dotsc,e_n\}$ be the standard basis  and $\allone=(1,\dotsc,1)\in \Q^n$.
  Then there exist sets $B_0,\dotsc,B_{k-1}\subset \Q^n$ satisfying
  \[ \abs{B_0}+\dotsb+\abs{B_{k-1}}=m  \ \text{ and } \
   kC\subset \bigcup \{ B_{i_1}+\dotsb+B_{i_k}+\Delta_i \allone: i_1 + \dots + i_k \equiv 1 \text{ mod } k \}, \]
 where $\Delta_i \eqdef \bigl\lfloor \frac{i_1+\dotsb+i_k}{k}\bigr\rfloor$ is
  an integer that depends only on $i=(i_1,\dotsc,i_k)$.
\end{lemma}
\begin{proof}
There will be three steps in the proof:
(1) restricting attention to a special class of $k$-fold sumsets $A$ in $\Q$,
(2) setting up simultaneous equations that describe possible ways to represent $A$ as a $k$-fold sumsets in~$\Z$,
(3) applying \Cref{lem:matrices} to deduce that there is a uniform way to represent such sets $A$
and translating this into the conclusion of the lemma.
 
 \paragraph{1. Restricting to sets \texorpdfstring{$A$}{A} of a special form.} 
  Since $kC+1=k(C+\frac{1}{k})$, every set of the form $kC+1$ is a $k$-fold sumset in $\Q$.
  So, by the assumption of the lemma, for every $n$-element
  $C\subset \Z$ there exists a set $B\subset \Z$ of size $m$ such that
  $kC+1\subset kB$. We restrict to sets of the form $C = \{kx: x \in C'\}$
  where $C' \subset \Z$ with $|C'|=n$. Letting
  \begin{gather}
   B_i\eqdef\{b\in \Z: kb+i\in B\}\quad \text{for each }i=0,1,\dotsc,k-1,\nonumber \\
    \ \text{ we have} \ \abs{B_0}+\dotsb+\abs{B_{k-1}}=m, \ \text{ and } \label{eq:sumsizes} \\
    kC'\subset \bigcup_{i_1+\dotsb+i_k\equiv 1\pmod k} B_{i_1}+\dotsb+B_{i_k}+\Delta_i. \label{eq:bigunion}
  \end{gather}
     
  \paragraph{2. Setting up the equations.}
  Consider the problem of deciding if,
  given an $n$-element $C'$, there are sets $B_0,\dotsc,B_{k-1}$ satisfying \eqref{eq:sumsizes}
  and \eqref{eq:bigunion}. There are $\binom{m+k-1}{k-1}$ ways to choose the sizes of $B_0,\dotsc,B_{k-1}$
  to satisfy \eqref{eq:sumsizes}. Treating the elements of $B_0,\dotsb,B_{k-1}$ as unknowns, for each of these ways,
  there are a number of potential ways that elements of $kC'$ can be written as elements of the right hand size of \eqref{eq:bigunion}.
  For each of these ways, we obtain a system of equations of the form
  \[
    c_{j_1}+\dotsb+c_{j_k}=b_{i_1}+\dotsb+b_{i_k}+\Delta_i
  \]
  as $j_1,\dotsc,j_k$ range over $[n]$, and $c_1,\dotsc,c_n$ denote the elements of~$C'$.
  Solvability of such a system is equivalent to being able to satisfy \eqref{eq:sumsizes} and \eqref{eq:bigunion}
  in a particular way.

  \paragraph{3. Conclusion of the proof.}
  By the assumption, for any choice of $c_1,\dotsc,c_n\in\Z$, at least one of these linear systems has
  a solution over~$\Z$. From \Cref{lem:matrices} it then follows that one of these linear systems has
  a solution $(b_1,\dotsc,b_m)\in\Q^m$ for every $(c_1,\dotsc,c_n)\in \Q^n$.
  We denote this system by $Lx + \Delta =Mc$, where $L$ is $m$-by-$N$, $M$ is $n$-by-$N$
  and $\Delta \in \Z^N$, with $N = |kC'|$ being the number of equations. 
  We let $(S_0,\dots,S_{k-1})$ be the partition of $\{1,\dots,m\}$
  where the unknowns $(x_j: j \in S_i)$ correspond to the choice of $B_i$.
  
  To create the vector model, we will solve this system for $c=e \eqdef (e_1,\dots,e_n) \in (\Q^n)^n$,
  recalling that $\{e_1,\dotsc,e_n\}$ is the standard basis of $\Q^n$.
  To do so, we consider the $n$-by-$m$ matrix whose columns 
  are the solutions $b^{(i)}$ of $Lb^{(i)} + \Delta = Me_i$.
  We denote its rows by $b_j\eqdef (b_j^{(1)},\dotsc,b_j^{(n)})\in \Q^n$.
  Then $b \eqdef (b_1,\dots,b_m)$ solves $Lx + \Delta \allone =Me$.
  
  The required sets $B_0,\dotsc,B_{k-1}$ for the lemma are $B_i \eqdef \{b_j: j \in S_i\}$.
 \end{proof}

\paragraph{Linear algebra.}

We are now nearly ready to prove the main result of this section.
As discussed in the introduction, the vector model established above
will allow us to bring into play methods from linear algebra.
To do so, we employ the following lemma, in which we call a linear space
a \emph{coordinate subspace} if it is spanned by some set of standard basis vectors.

\begin{lemma}\label{lem:coord}
  Let $F$ be a field and $V\subset F^n$ a vector space of codimension $d$. 
  Then there exists a coordinate subspace $W$ of dimension $d$ such that 
  $|W \cap (V+t)| \le 1$ for every $t\in F^n$.
  \end{lemma}
\begin{proof}
  The subspace $V$ is a zero set of a system of $d$ linear equations. 
  By putting the equations in a row echelon form and permuting the coordinates, 
  we may assume that the $i$'th linear equation is of the form
  $x_i+c_{i,i+1}x_{i+1}+\dotsb+c_{i,n}x_n=0$. 
  Then the subspace $W$ spanned by $e_1,\dotsc,e_d$ has the required property.
  Indeed, if we had two distinct vectors in $W \cap (V+t)$ for some $t \in F^n$ 
  then their difference would belong to $W \cap V$, but clearly $W \cap V = \{0\}$. 
\end{proof}

We conclude this section by proving its main result.

\begin{proof}[Proof of \Cref{thm:QvsZlowerbound}]
We will show that there exists a set $A\subset \Z$ satisfying
$\ell_{\Q}^{(k)}(A)\leq n$ and $\ell_{\Z}^{(k)}(A)\geq m\eqdef 2n-k^4n^{1-1/k}$.
We assume for contradiction that no such $A$ exists, thus obtaining the vector model
$B_0,B_1,\dotsc,B_{k-1}\subset \Q^n$ from  \Cref{lem:vectorization}.
For each $i=0,1,\dotsc,k-1$, let $V_i\subset \Q^m$ be the span of $B_i$. By \Cref{lem:coord} there is a coordinate subspace
$W_i\subset \Q^n$ of dimension at least $n-\abs{B_i}$ such that $\abs{V_i\cap (W_i+t)}\leq 1$ for all $t\in \Q^n$.
By shrinking $W_i$, we may assume that $\dim W_i=\max\{n-\abs{B_i},0\}$.

Let $C_i$ be the set of standard basis vectors that spans~$W_i$. 
Note that $\sum_{i=0}^{k-1} \abs{C_i}=\sum_{i=0}^{k-1} \dim W_i \ge kn-m$.
Let $C'$ be the set of standard basis vectors that are contained in at least $k-1$ of the sets $C_0,\dotsc,C_{k-1}$.
From $kn-m \le \sum_{i=0}^{k-1} \abs{C_i}\leq k\abs{C'}+(k-2)n$, and our choice of $m$, it follows
that $\abs{C'}\geq k^3 n^{1-1/k}$. Hence, there exists an index $i_0\in \{0,1,\dotsc,k-1\}$ such that
\[
  \abs{\bigcap_{i\neq i_0} C_i}\geq k^2n^{1-1/k}.
\]
Let $C''\eqdef \bigcap_{i\neq i_0} C_i$.
The $k$-fold sumset $kC''$ is of size
\begin{equation}\label{eq:Cprime_lowerbound}
  \abs{kC''}=\binom{k^2n^{1-1/k}+k-1}{k}\geq k^{2k} n^{k-1}/k!.
\end{equation}

Now consider $B\eqdef \bigcup_{i=0}^{k-1} B_i$ and 
write $D\eqdef \{\Delta_i \allone : i_1+\dotsb+i_k\equiv 1\pmod k\}$.
Noting that in every solution to $i_1+\dotsb+i_k\equiv 1\pmod k$
at least one of $i_1,\dotsc,i_k$ is not equal to $i_0$, we obtain
\begin{align*}
  kC'' \subset kC
   &  \subset \bigcup_{i_1+\dotsb+i_k\equiv 1\pmod k} B_{i_1}+\dotsb+B_{i_k}+\Delta_i \allone\\
    &\subset \bigcup_{i\neq i_0} B_i+(k-1)B+D.
\end{align*}
Now recall that each standard basis vector in $C''$ belongs to each $W_i$ with $i \ne i_0$.
Since $B_i\subset V_i$ and $\abs{W_i\cap (V_i+t)}\leq 1$ for every $t\in \Q^n$, it follows that
\begin{align*}
  \abs{kC''}&\leq \sum_{i\neq i_0}\abs{kC''\cap \bigl(B_i+(k-1)B+D\bigr)}
           \leq \sum_{i\neq i_0}\abs{(k-1)B+D}\\
           &\leq (k-1)\binom{m+k-2}{k-1}k^{k-1}
            \leq (k-1)\frac{(2n)^{k-1}}{(k-1)!}k^{k-1}.
\end{align*}
As this inequality contradicts \eqref{eq:Cprime_lowerbound}, we conclude that there does exist a set
$A$ satisfying $\ell_{\Q}^{(k)}(A)\leq n$ and $\ell_{\Z}^{(k)}(A)\geq 2n-k^4n^{1-1/k}$.
\end{proof}
  
\section{\texorpdfstring{$\N$}{N} versus \texorpdfstring{$\Z$}{Z}: \texorpdfstring{$2$-bases}{2-bases}}  \label{sec:NZ2}

Here we consider sets $A$ of natural numbers 
that have $2$-bases of size $n$ over the integers;
our task is to find small $2$-bases for $A$ over the natural numbers.
We start with the upper bound, implementing the dyadic construction
mentioned in the introduction.

\begin{proof}[Proof of \Cref{thm:NvsZupper2}]
By assumption, we have $A \subseteq B+B$ for some $B \subset \Z$ with $|B| \le n$.

We fix $x_1 < x_2 < \dots < x_n$ in $\N$ so that each element of $B$ 
is $\pm x_i$ for some $i$ and let 
\[ B^+ \eqdef (B \cap \N) \cup \bigcup_{j=0}^{\lfloor \log_2 n \rfloor} B_j,\]
where each $B_j$ contains all elements of the form $x_r - x_s$
where $s = 2^j \lfloor 2^{-j} r \rfloor$ or $r = 2^j \lceil 2^{-j} s \rceil$.

We claim that $A \subseteq B^+ + B^+$. To see this, consider $a \in A$
and write $a=b+b'$ with $b,b' \in B$. If both $b,b' \in \N$ then
we have the same representation $a=b+b'$ in $B^+ + B^+$.
The other case is that $a = x_r - x_t$ for some $1 \le t < r \le n$.
We write $r$ and $t$ in binary and let $2^j$ be the 
most significant column in which they differ.
Then $s \eqdef 2^j \lfloor 2^{-j} r \rfloor = 2^j \lceil 2^{-j} t \rceil$.
Now $a = x_r - x_t = (x_r - x_s) + (x_s - x_t)$ 
represents $a$ in $B^+ + B^+$, as required.
Clearly $|B^+| \le n + 2 n (1+\log_2 n)$.
\end{proof}

Now we prove the lower bound, where roughly speaking we find that one needs
order $n$ elements at each dyadic scale, in a suitably averaged sense.

\begin{proof}[Proof of \Cref{thm:NvsZlower2}]
We fix $0 < x_1 < x_2 < \dots < x_n$ in $\N$ rapidly growing, say each $x_r = 4^r$,
and let $A = \N \cap (C+C)$, where $C$ consists of all $\pm x_r$ with $1 \le r \le n$.
Then $C$ witnesses  $\ell_{\Z}(A)\leq 2n$. Consider any $B \subset \N$ with $A \subseteq B+B$.
We need a lower bound on $|B|$. Our bound will only use the fact that $B+B$ contains 
$A' \eqdef \{ x_r-x_t: 1 \le t < r \le n \}$, so we can discard from $B$ any elements~$\ge x_n$.

We partition $B$ into sets $B_r = B \cap [x_{r-1},x_r)$ for $1 \le r \le n$, where $x_0\eqdef 0$.
We further partition each $B_r$ into sets $B_{r,s} = B_r \cap [x_r-x_s,x_r-x_{s-1})$ for $1 \le s \le r$.
The key observations regarding this partition are that
if $a = x_r-x_t = b+b'$ for some $1 \le t < r \le n$ and $b,b' \in B$ with $b \ge b'$ then 
\begin{enumerate}
\item for some $1 \le s \le r$ we have $b \in B_{r,s}$ and $b'<x_s$, and 
\item either $b' \in B_s \cup B_{s-1}$ or $a = x_r-x_{s-1}$
is the only element of $A'$ of the form $b+b''$ with $b'' \in B$ and $b'' < x_{s-2}$,
in which case we call $a$ `unique'.
\end{enumerate}
Indeed, we must have $b \in B_r$, as if $b \ge x_r$ then $b+b'>x_r-x_t$ is too large
and if $b<x_{r-1}$ then $b+b' < 2x_{r-1} < x_r-x_t$ is too small.
Then $b \in B_{r,s}$ for some $1 \le s \le r$,
and $b' < x_s$, as otherwise $b+b' \ge x_r$ is too large.
Furthermore, if $b' < x_{s-2}$ then $x_r - x_s < b+b' < x_r-x_{s-1} + x_{s-2} < x_r - x_{s-2}$,
so $b+b' = x_r - x_{s-1}$.

For $1 \le s \le n$ we define degrees $d_s \eqdef |B_s| + |B_{s-1}| + \sum_{r:r\geq s} |B_{r,s}|$,
noting that $\sum_s d_s \le 3|B|$. 

For each $0 \le j \le \log_2 n$ we let $X_j \eqdef \{ x_s: d_s \in [2^j, 2^{j+1}) \}$,
so $\sum_j 2^j |X_j| \le \sum_s d_s \le 3|B|$.

For each $0 \le i \le \log_2 n$ we let $A_i \eqdef \{ x_r-x_t: r-t \in [2^i, 2^{i+1}) \}$.

For each $a = x_r-x_t \in A'$ we fix some representation $a = b_a+b'_a$  with $b_a\geq b'_a$ in $B+B$
and define the scale $i$ of $a$ 
by $a \in A_i$ and representation scale $j$ of $b_a$ by $x_s \in X_j$, where $b_a \in B_{r,s}$.
We note that $t \le s \le r$.

If $a$ is not unique, we call $a$ `up' if $j>i$ or `down' if $j \le i$. 
We define $n^!_i, n^+_i, n^-_i \in \Q$ so that $A_i$ 
has $n^!_i 2^i$ unique, $n^+_i 2^i$ up and $n^-_i 2^i$ down elements,
noting that $n^!_i + n^+_i + n^-_i = 2^{-i} |A_i| \ge n - 2^{i+1}$.

Clearly $\sum_i n^!_i \le \sum_i n^!_i 2^i \le |B|$,
as $a\mapsto b_a$ is an injective map from unique elements to~$B$.

To estimate the number of up elements, we fix any $j>i$ and any $x_s \in X_j$,
then note that there are $\le 2^{i+1}$ choices of $r \in [s,t+2^{i+1})\subset [s,s+2^{i+1})$ 
and then $\le 2^{i+1}$ choices of $t \in (r-2^{i+1},r]$. This gives 
\[ n^+_i 2^i \le \sum_{j: j>i} |X_j| \cdot  2^{i+1} \cdot 2^{i+1}. \]
Summing over $i$, we deduce
$\sum_i n^+_i \le 4 \sum_j 2^j |X_j| \le 12|B|$.

For down elements, for any $j \le i$ let $N_{i,j}$ count 
elements $a$ of scale $i$ with representation scale $j$, i.e. $N_{i,j}=|\{a\in A_i: b_a\in B_{r,s} \text{ and } x_s\in X_j\}|$.
Thus $n^-_i 2^i = \sum_{j: j \le i} N_{i,j}$. On the other hand,
by definition of degree and $X_j$ we have
\[ \sum_{i: i \ge j} N_{i,j} \le \sum_{x_s\in X_s} \abs{B_s\cup B_{s-1}}\cdot \abs[\bigg]{\bigcup_{r:r\geq s} B_{r,s}} \le |X_j| \cdot 2^{j+1} \cdot 2^{j+1}. \]
Summing over $j$, we deduce
$\sum_i n^-_i \le \sum_{j \le i} 2^{-i} N_{i,j} \le 4 \sum_j 2^j |X_j| \le 12|B|$.

To conclude, $25|B| \ge \sum_i (n^!_i + n^+_i + n^-_i) \ge \sum_i (n - 2^{i+1}) \ge n\log_2 n - 4n$.
\end{proof}

\section{\texorpdfstring{$\N$}{N} versus \texorpdfstring{$\Z$}{Z}: higher order bases}  \label{sec:NZk}

Our task here is to consider a set $A$ of natural numbers 
with a $k$-basis of size $n$ over the integers
and a small $k$-basis for $A$ over the natural numbers.
In fact, it will be more convenient to consider a set $A$ of non-negative reals
with a $k$-basis of size $n$ over the reals
and find a small $k$-basis $B$ over the set of non-negative reals.
This suffices, as if we thus obtain $B$ from some $A \subset \N$
then taking $\lfloor b \rfloor$ and $\lceil b \rceil$ for each $b \in B$
gives a $k$-basis of $A$ over $\N$, by the same rounding argument
as in the proof of \Cref{thm:QvsZupperbound}.
\Cref{iteratedNvsZthm} therefore follows from the following theorem.

\begin{theorem}\label{RvsR+}
  If $B \subset \R_{\geq 0}$ with $|B|=n$ then we can find a set $X\subset\R_{\geq 0}$ with $|X|\leq 2n^3k\log(k)$
  such that $\left(k (B\cup -B)\right)\cap \R_{\geq0}\subset kX$.
\end{theorem}

As discussed in the introduction, our approach is to find an arithmetic progression (AP) that  approximates $B$, 
effectively deal with all the big elements of $k(B\cup -B)$, quotient out by the step size of the AP 
to reduce the number of elements, and then iterate.

We first find the appropriate AP.

\begin{lemma}\label{APfindinglem}
Consider reals $0\leq x_1<\dots<x_n=1$ so that $x_1\leq 1-1/C$ with $C\in\mathbb{N}$.
Then there exist $L\geq C^{-n-1}$ and $y_i\in \mathbb{N}$ so that, writing $x_i=y_iL+z_i$, 
each $|z_i|\leq L/C$ and $z_1=z_n$.
\end{lemma}

\begin{proof}
Consider the map $q\colon\R\to(\R/\Z)^n;\lambda\mapsto ([\lambda x_i])_{i=1}^n$, 
where $[\lambda x_i]$ is the fractional part of $\lambda x_i$. 
Consider the subset $H\eqdef\{z\in (\R/\Z)^n: z_1=z_n\}$ and its preimage $q^{-1}(H)\subset\R$. 
Since $x_n/x_1\geq 1-1/C$, we find that $q^{-1}(H) = \alpha\Z$ is an AP with step size $\alpha\leq C$. 
By the pigeonhole principle there exist $i,j\leq \alpha C^n\leq C^{n+1}$ so that 
$\|q((i-j)\alpha)\|_\infty = \|q(i\alpha)-q(j\alpha)\|_\infty \leq 1/C$. 
Let $L\eqdef\frac{1}{(i-j)\alpha}=\Omega(C^{-n-1})$ and write $x_i=y_iL+z_i$ with $y_i\in \Z$ and $|z_i|\leq L/C$.
\end{proof}

The following lemma establishes the `separation of scale' mentioned in the introduction:
it shows that elements of $k (B\cup (-B))$ which are small compared to the scale of the AP
are expressible in $ k (B'\cup (-B'))$ where $B'$ is on the scale of the error terms in the approximation.

\begin{lemma}\label{smalldifcover}
Let $B\eqdef\{x_i\}_{i=1}^n$ where $x_i=y_iL+z_i$, with $y\in\Z$,  $|z_i|\leq L/C$ and $C>2k$.
Let $B'\eqdef\{z_i\}_{i=1}^n$. Then 
$$\{x\in k (B\cup (-B)):|x|< L/2\}\subset k (B'\cup (-B')).$$
\end{lemma}
\begin{proof}
Consider $\sum_{j=1}^k \epsilon_j x_{i_j}\in k (B\cup (-B))$ where $\epsilon_j\in\{\pm 1\}$. 
Note that $\left|\sum_{j=1}^k \epsilon_j z_{i_j}\right|\leq \sum_{j=1}^k  \left|z_{i_j}\right|\leq kL/C<L/2$. 
On the other hand, since $y_i\in\Z$, if $\left|\sum_{j=1}^k \epsilon_j y_{i_j}L\right|>0$, 
then in fact $\left|\sum_{j=1}^k \epsilon_j y_{i_j}L\right|\geq L$, so that $\left|\sum_{j=1}^k \epsilon_j x_{i_j}\right|>L/2$. 
Hence, if $\left|\sum_{j=1}^k \epsilon_j x_{i_j}\right|<L/2$, then $\left|\sum_{j=1}^k \epsilon_j y_{i_j}L\right|=0$, 
so that $\sum_{j=1}^k \epsilon_j x_{i_j}=\sum_{j=1}^k \epsilon_j z_{i_j}\in k (B'\cup (-B')) $.
\end{proof}

In the next lemma we efficiently cover the elements of $k (B\cup (-B))$ on the scale of the AP.

\begin{lemma}\label{bigdifcover}
 Let $B=\{x_1<\dots< x_n=1\}\subset\R_{\geq0}$ and $L>0$. 
 Then there exists a set $X\subset\R_{\geq0}$ with $|X|\leq nk\log_2(3k/L)$ such that 
 $$\{x\in k (B\cup (-B)):x\geq L/2\}\subset k X. $$
\end{lemma}

\begin{proof}
Fix $\sum_{j=1}^q  x_{i_j}-\sum_{j=q+1}^k  x_{i_j}\in k (B\cup (-B))$ with $\sum_{j=1}^q  x_{i_j}-\sum_{j=q+1}^k  x_{i_j}\geq L/2$. 

For $0 \le m \le \log(3/L + k)$ we consider the differences
$$D_m\eqdef\sum_{j=1}^q  \lfloor 2^m x_{i_j}\rfloor-\sum_{j=q+1}^k  \lceil 2^m x_{i_j}\rceil.$$
Since $x_i\in (0,1)$ implies $\lceil x_{i}\rceil=1$ and $\lfloor x_{i}\rfloor=0$, it follows that $D_0=1+q-k$ or $D_0=q-k$ depending on whether $n\in \{i_1,\dotsc,i_q\}$ or not.
In either case $D_0\leq 1$. Also, for any $m$ we have
$$\lfloor 2^m x_{i_j}\rfloor\leq 2^mx_i\leq \lceil 2^m x_{i_j}\rceil\leq \lfloor 2^m x_{i_j}\rfloor+1, $$ 
$$\ \text{ so } \ \ D_m\geq 2^m\Bigl(\sum_{j=1}^q  x_{i_j}-\sum_{j=q+1}^k  x_{i_j}\Bigr) - k.$$
Hence, for $m=\log(3/L + k)$, we have $D_m\geq 2^mL/2 - k >0$.

We claim that there exists an $m_0\in\{0,1,\dots,\log_2(3k/L)\}$ such that $D_{m_0}\in\{0,1,\dots, k-1\}$.
To see this, we may assume that $D_0<0$, for otherwise we are done, in view of $D_0\leq 1$.
Let $m_0$ be minimal so that $D_{m_0}\geq 0$. Note that $m_0\geq 1$, so that $D_{m_0-1}<0$ exists. 
For any number $x$, we have $\lfloor 2^{m_0} x\rfloor\leq 2\lfloor 2^{m_0-1} x\rfloor +1$ 
and $\lceil 2^{m_0} x\rceil\geq 2\lceil 2^{m_0-1} x\rceil-1$. 
Hence, $D_{m_0}\leq 2D_{m_0-1}+k$. Since $D_{m_0-1}\leq -1$, the claim follows.

With these observations in place, we can construct $X$ as follows. For $m\leq \log_2(3k/L)$ let
$$X_m\eqdef\left\{x_i- \frac{\left\lfloor2^m x_i\right\rfloor - p}{2^m}:1\leq i\leq n,0\leq p\leq k-1\right\}\cup\left\{ \frac{\left\lceil2^m x_i\right\rceil}{2^m}-x_i:i=1,\dots,n\right\},$$
and let $X\eqdef\bigcup_{m=0}^{\log_2(3k/L)} X_m$. First note that indeed $|X|\leq nk\log_2(3k/L)$. 
To prove that $k X$ indeed covers all big differences, 
consider any $\sum_{j=1}^q  x_{i_j}-\sum_{j=q+1}^k  x_{i_j}\in k (B\cup (-B))$ 
with $\sum_{j=1}^q  x_{i_j}-\sum_{j=q+1}^k  x_{i_j}\geq L/2$. 
Using the claim, find $m_0$ so that $2^{m_0}D_{m_0}\in \{0,\dots, k-1\}$. 
Note that $X_{m_0}\subset X$ contains the elements
$$y_{j}\eqdef\begin{cases}
x_{i_1}-\frac{\left\lfloor2^{m_0} x_{i_1}\right\rfloor - D_{m_0}}{2^{m_0}}&\text{if $j=1$}\\
 x_{i_j}-\frac{\left\lfloor2^{m_0} x_{i_j}\right\rfloor}{2^{m_0}} &\text{if $j=2,\dots  q$}\\
  \frac{\left\lceil2^{m_0} x_{i_j}\right\rceil}{2^{m_0}}-x_{i_j}&\text{if $j=q+1,\dots  k$},
\end{cases}$$
so that $\sum_{j=1}^k y_j= \sum_{j=1}^q  x_{i_j}-\sum_{j=q+1}^k  x_{i_j} $. The lemma follows.
\end{proof}

Having prepared the ingredients, we now come to the proof of the main result of this section.

\begin{proof}[Proof of \Cref{RvsR+}]
Set $C=3k$.
We proceed by induction on $n$. Write $B_n=B=\{x_1<\dots<x_n\}$.
First note that multiplying all elements by the same factor doesn't affect the problem, 
so we may assume $x_n=1$. We distinguish two cases:~either $x_1\geq 1-1/C$ or not.

If $x_1\geq 1-1/C$, we let $L=1$ and note that $x_i=y_i L+ z_i$ with $y_i=1$ for all $i$ and $|z_i|\leq L/C$. Hence, by \Cref{bigdifcover} there exists an $X^{n,1}\subset\R_{\geq0}$ with $|X^{n,1}|\leq nk\log_2(3k)\leq 2nk\log(k)$ such that
$$\{x\in k (B\cup (-B)):x\geq 1/2\}\subset k X^{n,1}.$$
By \Cref{smalldifcover}, we find that if we let $B'_n\eqdef\{|z_i|:i=1,\dots, n\}$, then
$$\{x\in k (B\cup (-B)):x< 1/2\}\subset k (B'_n\cup -B'_n),$$
so it remains to cover $k (B'_n\cup -B'_n)$. We write $B_n'=\{x'_1<\dots<x'_n\}$ 
and renormalise so that $x'_n=1$. Since $z_n=0$, we have $x'_1=0<1-1/C$, so we reduce to the next case.

For notational convenience, we henceforth assume $x_1<1-1/C$ (rather than $x_1'<1-1/C$). 
Apply \Cref{APfindinglem} to find $L\geq C^{-n-1}$ so that we can write $x_i=y_iL+z_i$ 
with $y_i\in \Z$, each $|z_i|\leq L/C$ and $z_1=z_n$. By \Cref{bigdifcover}, 
we can find $X^{n,2}$ with $|X^{n,2}|\leq nk\log_2(3k/L)\leq nk\log_2(3kC^{n+1})\leq 2n^2k\log(k)$ so that 
$$\{x\in k (B_n\cup (-B_n)):x\geq L/2\}\subset k X^{n,2}.$$
By \Cref{smalldifcover}, we find that if we let $B_{n-1}\eqdef\{|z_i|:i=1,\dots, n\}$ then
$$\{x\in k (B_n\cup (-B_n)):x< L/2\}\subset k (B_{n-1}\cup -B_{n-1}).$$
Since $z_1=z_n$ we have $|B_{n-1}|\leq n-1$. 
Hence, by induction, there is $Y\subset\R_{\geq0}$ of size $|Y|=2(n-1)^3k\log(k)$ such that 
$k (B_{n-1}\cup -B_{n-1})\subset k Y$.
Combining these sets, we obtain
$$k (B\cup -B)\subset k (Y\cup X^{n,1}\cup X^{n,2}).$$
Finally, $|Y\cup X^{n,1}\cup X^{n,2}|\leq 2(n-1)^3k\log(k)+2nk\log(k)+2n^2k\log(k)\leq 2n^3k\log(k)$.
\end{proof}

\section{Concluding remarks and open problems}

\paragraph{Bases over \texorpdfstring{$\Q$}{Q} versus \texorpdfstring{$\Z$}{Z}.}

Here we believe that the upper bound of $2n$ should be tight.
In the vector model, for the case $k=2$ one needs to prove the following.

\begin{conjecture} \label{conj1}
Suppose $B_0,B_1 \subset \Q^n$ are such that $B_0+B_1$
contains every $e_i + e_j$ with $1 \le i \le j \le n$.
Then $|B_0|+|B_1| \ge 2n$.
\end{conjecture}

An extension of \Cref{conj1} would be to show that
if $|B_0|=n-t$ for some $t \ge 0$ then $|B_1| \ge n + \tbinom{t+1}{2}$.
It is also natural to consider \Cref{conj1} over finite fields,
for which we expect the answer is essentially the same,
changing $2n$ to $2n-2$ over $\F_{2^m}$ (when $e_i+e_i=0$).

\paragraph{Bases over \texorpdfstring{$\N$}{N} versus \texorpdfstring{$\Z$}{Z}.}

Is there is any polynomial gap for fixed $k$? We believe not.

\begin{conjecture} \label{conj2}
For any $k \in \N$ and $\veps>0$ there is $n_0 \in \N$ so that
if $A \subseteq \N \cap kB$ for some $B \subset \Z$ with $|B|=n>n_0$
then $A \subseteq kC$ for some $C \subset \N$ with $|C| \le n^{1+\veps}$.
\end{conjecture}

\paragraph{Multiplication / higher dimensions.}

The change of domain question can be formulated in any abelian group.
For example, we could consider $G=\Z^2$, $D=\N^2$ 
and compare bases over $A \subset D$ over $D$ and over $G$.
The infinite-dimensional version of this question is even more natural,
as it is equivalent to taking $G = \R^\times_{>0}$ (the positive reals with multiplication)
and comparing multiplicative bases of  $A \subset \N_{>0}$ over $G$ and over $\N_{>0}$.
Here, we believe there might be a polynomial gap.

\begin{question} \label{conj3}
Are there $\veps>0$ and $n_0 \in \N$ so that for any $n>n_0$
there is $A \subseteq \N \cap BB$ for some $B \subset \R_{>0}$ with $|B|=n$
such that $A \not\subseteq CC$ for any $C \subset \N$ with $|C|<n^{1+\veps}$?
\end{question}


\bibliographystyle{unsrt}
\bibliography{ref.bib}

\end{document}